\documentclass{amsart}
\usepackage{geometry,amssymb,amsmath}
\usepackage{color} 
 
\newcommand{\refe}[1]{(\ref{#1})}
\newcommand{\dst}{\displaystyle}
\newcommand{\RR}{{\mathbb R} }

\newcommand{\CC}{{\mathbb C}}

\renewcommand{\u}{\mbox{${\textbf{u}}$}}
\newcommand{\newu}{\mbox{$\widetilde{\textbf{u}}$}}

\newcommand{\K}{\mathrm{K}}
\newcommand{\Ku}{\mathrm{K}^{\alpha,\beta}}

\newcommand{\tilG}{\widetilde{\Gamma}_q}

\newcommand{\bq}{\begin{equation}}
\newcommand{\eq}{\end{equation}}
\newcommand{\ba}{\begin{array}}
\newcommand{\ea}{\end{array}}

\newtheorem{theorem}{Theorem}

\begin{document}

\title{On the limit of non-standard $q$-Racah polynomials}

\author{R. \'{A}lvarez-Nodarse}
\address{IMUS \& Departamento de An\'alisis Matem\'atico, Universidad de Sevilla.
Apdo. 1160, E-41080  Sevilla, Spain}\email{ran@us.es}

\author{R. Sevinik Ad\i g{\"{u}}zel}
\address{Department of Mathematics, Faculty of Science, Sel\c{c}uk University, 42075,
Konya, Turkey}\email{sevinikrezan@gmail.com}

\keywords{q-Racah polynomials, big $q$-Jacobi polynomials, dual $q$-Hahn polynomials, $q$-Hahn polynomials, limit relations}
\subjclass[2000]{33D45, 33C45, 42C05}


\begin{abstract}
The aim of this article is to study the limit transitions
from non-standard $q$-Racah polynomials to big $q$-Jacobi, dual $q$-Hahn, and 
$q$-Hahn polynomials such that the orthogonality properties and the
three-term recurrence relations remain valid. 
\end{abstract}

\maketitle

\section{Introduction}
The Askey-Wilson polynomials and $q$-Racah polynomials are the
most general classical orthogonal families from which all the other 
$q$-hypergeometric orthogonal polynomials can be obtained
by (possibly successive) limit transitions. 
There are several ways of getting these limits, but must of
them are not good enough by means of the orthogonality property
or the three term recurrence relation as it was pointed out by
Koornwinder in \cite{TK} for the case of $q$-Racah polynomials and 
big $q$-Jacobi polynomials. In fact, in \cite{TK} the author  studied 
the limit relation from the standard $q$-Racah polynomials defined
on the lattice \cite[page 422]{ks} $x(s)=q^{-s}+\gamma\delta q^{s+1}$ 
to the big $q$-Jacobi polynomials such that the orthogonality properties 
remain valid.

In this article, we introduce some limit formulas from the non-standard 
$q$-Racah polynomials $u_n^{\alpha,\beta}(x(s),a,b)$ \cite{RYR,NSU} 
defined on the lattice $x(s)=[s]_q[s+1]_q$ where $[s]_q$ are the symmetric 
$q$-numbers 
$$
[s]_q=\dst\frac{q^{s/2}-q^{-s/2}}{q^{1/2}-q^{-1/2}}, \quad s\in\mathbb{C},
$$ 
to the big $q$-Jacobi polynomials. 
Let us point out that the lattice for the non-standard polynomials is more appropriate 
for numerical analysis since it does not depend on any parameters  of the polynomials. 

Furthermore, we  consider the similar limit properties between the non-standard 
$q$-Racah-Krall-type polynomials \cite{RR} and the big $q$-Jacobi-Krall-type  polynomials
\cite{RC}. The Krall-type polynomials are polynomials which are orthogonal with respect to a
linear functional $\newu$ obtained from a
quasi-definite functional $\u:\mathbb{P}\mapsto\CC$  ($\mathbb{P}$, denotes the space
of complex polynomials with complex coefficients) via the addition of delta Dirac measures, 
i.e., $\newu$ is the linear functional 
$$\newu=\u+\sum_{k=1}^N A_k\delta_{x_k}, 
$$
where $A_k\in\RR$, $x_1,\ldots,x_k \in \mathbb{R}$
and $\delta_a$ is the  delta Dirac functional at the point $a$, i.e., 
$\langle \delta_a, p\rangle=p(a)$, where $p\in\mathbb{P}$.  

These kind of modifications firstly appeared as eigenfunctions of a fourth order 
linear differential operator with polynomial coefficients
that do not depend on the degree of the polynomials (see \cite{hkrall} or the 
more recent reviews \cite{a-nmp} and \cite[chapter XV]{akrall}).

Our main aim in this note is to obtain the limit formulas between non-standard 
$q$-Racah polynomials and big $q$-Jacobi, dual $q$-Hahn, and $q$-Hahn polynomials, respectively,
as well as between the corresponding Krall-type polynomials. In fact, the explicit limits 
from $q$-Racah polynomials to these other families are given in 
\cite[(14.2.15), (14.2.19) and (14.2.16,17,18)]{ks}, respectively. However, the formula
\cite[(14.2.15)]{ks}, for example, is not valid for getting the 
orthogonality of the resulting polynomials as it is pointed out by 
Koornwinder in \cite{TK}. In \cite{TK} Koornwinder proposed 
a more convenient limit formula that allows him to show that the support 
of the measure for the $q$-Racah polynomials transforms into the 
measure of the big $q$-Jacobi polynomials. However, in \cite{TK} it is not shown how the orthogonality 
relation and the three-term recurrence relation (TTRR) of the  $q$-Racah polynomials
transform into the big $q$-Jacobi ones. We fill this gap in this paper but using the aforementioned
non-standard $q$-Racah polynomials.

Following \cite{TK}, we present an alternative limit formula (to the one
suggested in \cite{NU93}) from the non-standard $q$-Racah to big $q$-Jacobi 
polynomials and from the non-standard $q$-Racah-Krall-type polynomials to big $q$-Jacobi-Krall-type 
polynomials such that the orthogonality property remains present. 
In particular, we show that by taking a proper limit, not 
only the polynomials $u_n^{\alpha,\beta,A,B}(x(s),a,b)$ and 
$u_n^{\alpha,\beta}(x(s),a,b)$ become into 
$P^{\alpha,\beta,A,B}_{n}(x,\widetilde{a},\widetilde{b},\widetilde{c};q)$ and $P^{\alpha,\beta}_{n}(x,\widetilde{a},\widetilde{b},\widetilde{c};q)$,
but also that the orthogonality relation and the three-term recurrence relation
(TTRR) of the  $q$-Racah polynomials become into the ones of big $q$-Jacobi polynomials.

The structure of the paper is as follows. 
We start by introducing some preliminary results 
on the $q$-Hahn polynomials and on the non-standard 
$q$-Racah polynomials and big $q$-Jacobi and Krall-type polynomials obtained 
via the addition of two mass points to the weight function of the 
these two polynomials in the forthcoming section 2. 
In section 3, we deal with the limit relation between the non-standard $q$-Racah polynomials and big $q$-Jacobi polynomials in detail. 
In section 4 
we consider the limits from the non-standard $q$-Racah polynomials 
to dual $q$-Hahn and $q$-Hahn polynomials, respectively.
Finally, in section 5 the limit between the corresponding Krall-type families are considered.

\section{Preliminary}
Here we include some properties of the non-standard $q$-Racah polynomials,
non-standard $q$-Racah-Krall-type polynomials with two mass points, big $q$-Jacobi polynomials, 
and $q$-Hahn monic polynomials. In the following, and throughout  
the paper, we denote $\kappa_q$ by the quantity 
$$
\kappa_q:=q^{1/2}-q^{-1/2}.
$$

The  non-standard monic $q$-Racah polynomials are defined by the following hypergeometric representation \cite{RYR,NSU} 
(for the definition and properties of basic series see \cite{GR})
\begin{equation}\label{q-racah}
\begin{split}
u^{\alpha,\beta}_n(s)_q:=u^{\alpha,\beta}_{n}(\mu(s),a,b)_q & = q^{-\frac n2(2a+1)}
\frac{(q^{a-b+1},q^{\beta+1},q^{a+b+\alpha+1};q)_n}
{\kappa_q^{2n}(q^{\alpha+\beta+n+1};q)_n}\times\\
&{}_{4}\varphi_3 \left(\ba{c} q^{-n},q^{\alpha+\beta+n+1}, q^{a-s},
q^{a+s+1}  \\ q^{a-b+1},q^{\beta+1},q^{a+b+\alpha+1}  \ea
\,\bigg|\, q \,,\, q \right),
\end{split}
\end{equation}
which are polynomials on the $q$-quadratic lattice of the form
\bq\label{qracahlattice}
\mu(s)=[s]_q[s+1]_q=c_1(q^s+q^{-s-1})+c_3,\quad 
c_1=q^{1/2}\kappa_q^{-2},\,\, c_3=-q^{-1/2}(1+q)\kappa_q^{-2}.
\eq
For $-\frac 12<a\leq b-1$, $\alpha>-1$, $-1<\beta<2a+1$, 
they satisfy the orthogonality relation ($b-a=N\in\mathbb{N}$)
\bq\label{ort-umod}
\begin{split}
\sum_{s = a }^{b-1} {u}_n^{\alpha, \beta}(s)_q {u}_m^{\alpha, \beta}(s)_q 
\rho(s)\Delta \mu(s-\mbox{$\frac 12$})= \delta_{n,m} d_n^2, \quad \Delta 
\mu(s-\mbox{$\frac 12$})=[2s+1]_q, 
\end{split}
\eq
where 
\[
d_n^2\!=\!\frac{(q; q)_n(q^{\alpha+1}; q)_n(q^{\beta+1}; q)_n(q^{b-a+\alpha+\beta+1}; q)_n
(q^{a+b+\alpha+1}; q)_n(q^{a-b+1}; q)_n(q^{\beta-a-b+1}; q)_n}{(q^{\frac12}-q^{-\frac12})^{4n}
(q^{\alpha+\beta+2}; q)_{2n}
(q^{\alpha+\beta+n+1}; q)_n},
\]
and $\rho$ is the weight function
(see Table 1\footnote{We have chosen $\rho(s)$ in
such a way that $\sum_{x = a }^{b-1} \rho(s)\mu(s-\mbox{$\frac 12$})=1$, i.e., to be a probability measure.} in \cite{RR}) of the non-standard $q$-Racah polynomials
\bq\label{qracahrho}
\begin{split}
\rho(s)= & \frac{\tilG(s\!+\!a\!+\!1) \tilG(s\!-\!a\!+\!\beta\!+\!1)
\tilG(s\!+\!\alpha\!+\!b\!+\!1)\tilG(b\!+\!\alpha\!-\!s)\tilG(\alpha\!+\!\beta\!+\!2)
}{ \tilG(s\!-\!a\!+\!1) \tilG(s\!+\!b\!+\!1)
\tilG(s\!+\!a\!-\!\beta\!+\!1)\tilG(b\!-\!s)\tilG(\alpha\!+\!1)\tilG( \beta\!+\!1)}\\
&\times\frac{\tilG(b\!-\!a) \tilG(a\!+\!b\!-\!\beta)}{\tilG(b\!-\!a\!+\!\alpha\!+\!\beta\!+\!1)\tilG(a\!+\!b\!+\!\alpha\!+\!1)},
\end{split}
\eq
where $\tilG(x)$, introduced in \cite[Eq. (3.2.24)]{NSU}, is related to the classical
$q$-Gamma function, $\Gamma_q$, \cite{ks} by formula 
$$
\tilG(s)= q^{-\frac{(s-1)(s-2)}{4}}\Gamma_q(s)=q^{-\frac{(s-1)(s-2)}{4}}
(1-q)^{1-s}\frac{(q;q)_\infty}{(q^s;q)_\infty},\quad 0<q<1.
$$
The non-standard $q$-Racah polynomials satisfy the TTRR \cite{RYR}
\begin{equation}\label{q-racahTTRR}
\mu(s)u_n^{\alpha,\beta}(s)_q=u_{n+1}^{\alpha,\beta}(s)_q+\beta_nu_{n}^{\alpha,\beta}(s)_q
+\gamma_nu_{n-1}^{\alpha,\beta}(s)_q,
\end{equation} 
\begin{equation*}
\begin{split}
\beta_n\!=& \![a]_q[a+1]_q\!-\!\frac{[\alpha+\beta+n+1]_q[a-b+n+1]_q[\beta+n+1]_q[a+b+\alpha+n+1]_q}
{[\alpha+\beta+2n+1]_q[\alpha+\beta+2n+2]_q}\\
&-\frac{[\alpha+n]_q[b-a+\alpha+\beta+n]_q[-a-b+\beta+n]_q[n]_q}
{[\alpha+\beta+2n]_q[\alpha+\beta+2n+1]_q},\\
\gamma_n=&\frac{[n]_q[\alpha+\beta+n]_q[\alpha+n]_q[\beta+n]_q[b-a+\alpha+\beta+n]_q[-a-b+\beta+n]_q}{[\alpha+\beta+2n-1]_q([\alpha+\beta+2n]_q)^2[\alpha+\beta+2n+1]_q}\\
&\times[a-b+n]_q[a+b+\alpha+n]_q.
\end{split}
\end{equation*}

 The $q$-Racah-Krall-type polynomials with two mass points are orthogonal with respect to a
linear functional $\newu$ obtained from a
quasi-definite functional $\u:\mathbb{P}\mapsto\CC$  ($\mathbb{P}$, denotes the space
of complex polynomials with complex coefficients) via the addition of delta Dirac measures, 
i.e., $\newu$ is the linear functional 
$$\newu=\u+A\delta_{a}+B\delta_{b-1}, 
$$
where $A,B\in\RR$, 
$\delta_a$ and $\delta_{b-1}$ are the  delta Dirac functionals at the point $a$ and $b-1$, i.e., 
$\langle \delta_a, p\rangle=p(a)$ and $\langle \delta_{b-1}, p\rangle=p(b-1)$, where $p\in\mathbb{P}$ and the kernel
\begin{equation}\label{rez13}
\begin{array}{l}
\K_n^{\alpha,\beta}(s_1, s_2):=\dst\sum_{k=0}^{n}{\dst\frac{u_k^{\alpha,\beta}(s_1)_qu_k^{\alpha,\beta}(s_2)_q}{d_k^2}}=\dst\frac{\alpha_n}{d_n^2}
\frac{u_{n+1}^{\alpha,\beta}(s_1)_q u_n^{\alpha,\beta}(s_2)_q-u_{n+1}^{\alpha,\beta}(s_2)_qu_n^{\alpha,\beta}(s_1)_q}
{x(s_1)-x(s_2)}.
\end{array}
\end{equation}
They satisfy the following orthogonality
relation
\bq\label{ort-umodi}
\begin{split}
\sum_{s = a }^{b-1} {u}_n^{\alpha, \beta, A,B}(s)_q {u}_m^{\alpha, \beta, A,B}(s)_q &
\rho(s)[2s+1]_q  + A {u}_n^{\alpha, \beta, A,B}(a)_q
{u}_m^{\alpha, \beta, A,B}(a)_q \\ & +B{u}_n^{\alpha, \beta, A,B}(b\!-\!1)_q
{u}_m^{\alpha, \beta, A,B}(b\!-\!1)_q = \delta_{n,m} (d_n^{A,B})^2,
\end{split}
\eq
where $\rho$ is the non-standard $q$-Racah weight function (see Table 1 in \cite{RR}\footnote{We have chosen $\rho(s)$ in such a way that
$\sum_{s = a }^{b-1} \rho(s)[2s+1]_q=1$, i.e., to be a probability measure.}). 
They can be written as \cite{RR} 
\begin{equation}
u_n^{\alpha,\beta,A,B}(s)_q=u_n^{\alpha,\beta}(s)_q-Au_n^{\alpha,\beta,A,B}(a)_q\K_{n-1}^{\alpha,\beta}(s,a)-Bu_n^{\alpha,\beta,A,B}(b-1)_q\K_{n-1}^{\alpha,\beta}(s,b-1).
\label{repforqracah}
\end{equation}
Moreover,  the following expressions for the
 the values at the points $s=a$ and $s=b-1$ and the norm $(d_n^{A,B})^2$ of the modified
polynomials ${u}_n^{\alpha, \beta, A,B}(s)_q$, respectively, hold
\begin{equation}\label{u-u2m}
\begin{array}{l}\dst
{u}_n^{\alpha, \beta, A, B}(a)_q=\frac{(1+B \Ku_{n-1}(b\!-\!1, b\!-\!1))u_n^{\alpha, \beta}(a)_q-
B\Ku_{n-1}(a, b\!-\!1)u_n^{\alpha, \beta}(b\!-\!1)_q}{\kappa_{n-1}^{\alpha, \beta}(a,b\!-\!1)},\\
{u}_n^{\alpha, \beta, A, B}(b\!-\!1)_q=\dst
\frac{-A\Ku_{n-1}(b\!-\!1, a)u_n^{\alpha, \beta}(a)_q+(1+A\Ku_{n-1}(a, a))u_n^{\alpha, \beta}(b\!-\!1)_q}
{\kappa_{n-1}^{\alpha, \beta}(a,b\!-\!1)},
\end{array}
\end{equation}\begin{small}
\begin{equation*}\label{dn-u2m}
\begin{split}
(d_n^{A,B})^2  =
d_n^2 &\! + \!\frac{A ({u}_n^{\alpha, \beta}(a)_q)^2 \{1\!+\! B\Ku_{n-1}(b\!-\!1,b\!-\!1)\}  +
B({u}_n^{\alpha, \beta}(b\!-\!1)_q)^2 \{ 1\!+\! A\Ku_{n-1}(a,a) \} }{\kappa_{n-1}^{\alpha, \beta}(a, b\!-\!1)} \\
& -\frac{2 AB {u}_n^{\alpha, \beta}(a)_q {u}_n^{\alpha, \beta}(b\!-\!1)_q\Ku_{n-1}(a,b\!-\!1)}
{\kappa_{n-1}^{\alpha, \beta}(a, b\!-\!1)},
\end{split}
\end{equation*}\end{small}
where
\begin{equation}\label{K(a,b)}\begin{split}
\kappa_{m}^{\alpha, \beta}(s,t)& =1+A\Ku_{m}(s,s)+ B\Ku_{m}(t,t)\\
& +AB\left\{\Ku_{m}(s,s)\Ku_{m}(t,t)-(\Ku_{m}(s,t))^2\right\},\end{split}
\end{equation}
where $\K_{m}^{\alpha, \beta}(s,t)$ are the kernels defined by \refe{rez13}
and  $d_n^2$ denotes the squared norm of the $n$-th non-standard $q$-Racah polynomials
(see Table 1 in \cite{RR}).

The non-standard $q$-Racah-Krall-type polynomials satisfy the TTRR \cite{RR}
\begin{equation}\label{q-racahmodTTRR}
\mu(s)u_n^{\alpha,\beta,A,B}(s)_q=\alpha_n^{A,B}u_{n+1}^{\alpha,\beta,A,B}(s)_q+\beta_n^{A,B}u_{n}^{\alpha,\beta,A,B}(s)_q
+\gamma_n^{A,B}u_{n-1}^{\alpha,\beta,A,B}(s)_q,
\end{equation} 
\begin{small}
\begin{equation*}
\begin{split}
\alpha_n^{A,B} &= 1, \\
\beta_n^{A,B} &=   \beta_n
-A \left(\frac{{u}_n^{\alpha, \beta, A,B}(a)_q u_{n-1}^{\alpha, \beta }(a)_q}{{d}_{n-1}^2}-
\frac{{u}_{n+1}^{\alpha, \beta, A,B}(a)_qu_{n}^{\alpha, \beta }(a)_q }{{d}_{n}^2}\right) \\
&-B \left(\frac{{u}_n^{\alpha, \beta, A,B}(b-1)_q u_{n-1}^{\alpha, \beta }(b-1)_q }{{d}_{n-1}^2}-
\frac{{u}_{n+1}^{\alpha, \beta, A,B}(b-1)_q u_{n}^{\alpha, \beta }(b-1)_q }{{d}_{n}^2}\right), \\
\gamma_n^{A,B} &= \gamma_n\frac{1\!+\!\Delta_n^{A,B}}{1\!+\!\Delta_{n-1}^{A,B}},\,\, \Delta_n^{A,B}=
\frac{ A {u}_n^{\alpha, \beta, A,B}(a)_q u_{n}^{\alpha, \beta }(a)_q }{{d}_{n}^2}\!+\!
\frac{B {u}_n^{\alpha, \beta, A,B}(b\!-\!1)_q u_{n}^{\alpha, \beta }(b\!-\!1)_q}{{d}_{n}^2}.
\end{split}
\end{equation*}
\end{small}


The monic big $q$-Jacobi polynomials are defined by the following basic series \cite{ks}
\begin{equation}\label{bigqjacobi}
\begin{split}
P^{\alpha,\beta}_{n}(z,\widetilde{a},\widetilde{b},\widetilde{c};q) & = (-\widetilde{a})^nq^{\frac{n(n+1)}{2}}\frac{(q\widetilde{b},q\widetilde{c};q)_n}
{(\widetilde{a}\widetilde{b}q^{n+1};q)_n}
{}_{3}\varphi_2 \left(\!\!\!\ba{c} q^{-n},\widetilde{a}\widetilde{b}q^{n+1}, q\widetilde{c}z^{-1}
\\ q\widetilde{b},q\widetilde{c}  \ea
\,\bigg|\, q \,,\, \frac{z}{\widetilde{a}} \!\!\right).
\end{split}
\end{equation}
They satisfy the orthogonality relation
\begin{equation}\label{bigqjacobiorth}
\begin{split}
&\int_{\widetilde{c}q}^{\widetilde{a}q}
P_m(z;\widetilde{a}, \widetilde{b}, \widetilde{c}; q)P_n(z; \widetilde{a}, \widetilde{b}, \widetilde{c}; q)\widetilde{\rho}(z)
d_q z=\int_{\widetilde{c}q}^0[.]d_qz+\int_{0}^{\widetilde{a}q}[.]d_qz\\
&=(1-q)(-\widetilde{c})\sum_{s=0}^{\infty}P_n(\widetilde{c}q^{s+1})P_m(\widetilde{c}q^{s+1})\widetilde{\rho}(\widetilde{c}q^{s+1})q^{s+1}\\
&+(1-q)\widetilde{a}\sum_{s=0}^{\infty}P_n(\widetilde{a}q^{s+1})P_m(\widetilde{a}q^{s+1})\widetilde{\rho}(\widetilde{a}q^{s+1})q^{s+1}
= \widetilde{d}_n^2\delta_{mn}
\end{split}
\end{equation}
for $0<q\widetilde{a}<1, 0\leq q\widetilde{b}<1, \widetilde{c}<0$, where
\begin{equation}\label{rhodn}
\begin{split}
\widetilde{\rho}(z)&=\frac{(\widetilde{a}^{-1}z, \widetilde{c}^{-1}z; q)_{\infty}}{(z, \widetilde{b}\widetilde{c}^{-1}z;q)_{\infty}}
\frac{(\widetilde{a}q, \widetilde{b}q, \widetilde{c}q, \widetilde{a}\widetilde{b}\widetilde{c}^{-1}q; q)_{\infty}}
{\widetilde{a}q(1-q)(q, \widetilde{a}\widetilde{b}q^2, \widetilde{a}^{-1}\widetilde{c}, \widetilde{a}\widetilde{c}^{-1}q; q)_{\infty}},\\
\widetilde{d}_n^2&=\frac{1-\widetilde{a}\widetilde{b}q}{1-\widetilde{a}\widetilde{b}q^{2n+1}}\frac{(q, \widetilde{b}q, \widetilde{a}q,\widetilde{c}q, \widetilde{a}\widetilde{b}\widetilde{c}^{-1}q; q)_n}{(\widetilde{a}\widetilde{b}q, \widetilde{a}\widetilde{b}q^{n+1},\widetilde{a}\widetilde{b}q^{n+1}; q)_n}(-\widetilde{a}\widetilde{c}q^2)^nq^{(^n _2)},
\end{split}
\end{equation}
and $\int_{0}^{t} f(z)d_qz$ is the $q$-Jackson integral \cite{GR}.

The monic big $q$-Jacobi polynomials satisfy the following TTRR  
\begin{equation*}
z P_n(z,\widetilde{a},\widetilde{b},\widetilde{c};q)=P_{n+1}(z,\widetilde{a},\widetilde{b},\widetilde{c};q)+
\widetilde{\beta}_nP_{n}(z,\widetilde{a},\widetilde{b},\widetilde{c};q)
+\widetilde{\gamma}_nP_{n-1}(z,\widetilde{a},\widetilde{b},\widetilde{c};q)
\end{equation*} 
where 
\begin{equation}\label{tildebetan}
\begin{split}
\widetilde{\beta}_n&\!=\!1\!\!-\!\frac{(1-\widetilde{a}q^{n+1})(1-\widetilde{a}\widetilde{b}q^{n+1})(1-\widetilde{c}q^{n+1})}
{(1-\widetilde{a}\widetilde{b}q^{2n+1})(1-\widetilde{a}\widetilde{b}q^{2n+2})}
\!+\!\widetilde{a}\widetilde{c}q^{n+1}\frac{(1-q^{n})(1-\widetilde{a}\widetilde{b}\widetilde{c}^{-1}q^{n})(1-\widetilde{b}q^{n})}{(1-\widetilde{a}\widetilde{b}q^{2n})(1-\widetilde{a}\widetilde{b}q^{2n+1})},\\
\widetilde{\gamma}_n&=-\widetilde{a}\widetilde{c}q^{n+1}\frac{(1-q^{n})(1-\widetilde{a}q^{n})(1-\widetilde{b}q^{n})(1-\widetilde{c}q^{n})(1-\widetilde{a}\widetilde{b}q^{n})(1-\widetilde{a}\widetilde{b}\widetilde{c}^{-1}q^{n})}
{(1-\widetilde{a}\widetilde{b}q^{2n-1})(1-\widetilde{a}\widetilde{b}q^{2n})^2(1-\widetilde{a}\widetilde{b}q^{2n+1})}.
\end{split}
\end{equation}  

The big $q$-Jacobi-Krall-type polynomials with two mass points  satisfy the orthogonality
relation
\begin{equation}\label{bigqjacobimodiorth}
\begin{split}
&\int_{\widetilde{c}q}^{\widetilde{a}q}
P_m^{A,B}(z;\widetilde{a}, \widetilde{b}, \widetilde{c}; q)P_n^{A,B}(z; \widetilde{a}, \widetilde{b}, \widetilde{c}; q)\widetilde{\rho}(z)
d_q z + A P_n^{A,B}(\widetilde{c}q;\widetilde{a}, \widetilde{b}, \widetilde{c}; q)
P_m^{A,B}(\widetilde{c}q;\widetilde{a}, \widetilde{b}, \widetilde{c}; q) \\ & +BP_n^{A,B}(\widetilde{a}q;\widetilde{a}, \widetilde{b}, \widetilde{c}; q)_q
P_m^{A,B}(\widetilde{a}q;\widetilde{a}, \widetilde{b}, \widetilde{c}; q) =(1-q)(-\widetilde{c})\sum_{s=0}^{\infty}P_n^{A,B}(\widetilde{c}q^{s+1})P_m^{A,B}(\widetilde{c}q^{s+1})\widetilde{\rho}(\widetilde{c}q^{s+1})q^{s+1}\\
&+(1-q)\widetilde{a}\sum_{s=0}^{\infty}P_n^{A,B}(\widetilde{a}q^{s+1})P_m^{A,B}(\widetilde{a}q^{s+1})\widetilde{\rho}(\widetilde{a}q^{s+1})q^{s+1}
+ A P_n^{A,B}(\widetilde{c}q;\widetilde{a}, \widetilde{b}, \widetilde{c}; q)
P_m^{A,B}(\widetilde{c}q;\widetilde{a}, \widetilde{b}, \widetilde{c}; q) \\
&+BP_n^{A,B}(\widetilde{a}q;\widetilde{a}, \widetilde{b}, \widetilde{c}; q)
P_m^{A,B}(\widetilde{a}q;\widetilde{a}, \widetilde{b}, \widetilde{c}; q)
= \delta_{n,m} (\widetilde{d}_n^{A,B})^2
\end{split}
\end{equation}
where $\widetilde{\rho}(z)$ is the big $q$-Jacobi weight function (see Table 1 in \cite{RC}\footnote{We have chosen $\widetilde{\rho}(z)$ in such a way that
$\int_{s = \widetilde{c}q }^{\widetilde{a}q} \widetilde{\rho}(z)d_qz=1$, i.e., to be a probability measure.})
and
\begin{equation}\label{norm}
\begin{array}{l}\dst
P_n^{A,B}(\widetilde{c}q;\widetilde{a}, \widetilde{b}, \widetilde{c}; q)=\frac{(1+B \K_{n-1}(\widetilde{a}q, \widetilde{a}q))P_n(\widetilde{c}q;\widetilde{a}, \widetilde{b}, \widetilde{c}; q)-
B\K_{n-1}(\widetilde{c}q, \widetilde{a}q)P_n(\widetilde{a}q;\widetilde{a}, \widetilde{b}, \widetilde{c}; q)}{\kappa_{n-1}(\widetilde{c}q,\widetilde{a}q)},\\
P_n^{A,B}(\widetilde{a}q;\widetilde{a}, \widetilde{b}, \widetilde{c}; q)=\dst
\frac{-A\K_{n-1}(\widetilde{a}q, \widetilde{c}q)P_n(\widetilde{c}q;\widetilde{a}, \widetilde{b}, \widetilde{c}; q)+(1+A\K_{n-1}(\widetilde{c}q, \widetilde{c}q))P_n(\widetilde{c}q;\widetilde{a}, \widetilde{b}, \widetilde{c}; q)}
{\kappa_{n-1}(\widetilde{c}q,\widetilde{a}q)},
\end{array}
\end{equation}\begin{small}
where 
\begin{equation}\label{kerbigqjac}
\begin{array}{l}
\K_n(z_1, z_2):=\dst\sum_{k=0}^{n}{\dst\frac{P_n(s_1,\widetilde{a},\widetilde{b},\widetilde{c};q)P_n(s_2,\widetilde{a},\widetilde{b},\widetilde{c};q)}{\widetilde{d}_k^2}},
\end{array}
\end{equation}
are the corresponding kernels. Moreover,
\begin{equation*}
\begin{split}
(\widetilde{d}_n^{A,B})^2  =
\widetilde{d}_n^2 &\! + \!\frac{A (P_n(\widetilde{c}q;\widetilde{a}, \widetilde{b}, \widetilde{c}; q))^2 \{1\!+\! B\K_{n-1}(\widetilde{a}q,\widetilde{a}q)\}  +
B(P_n(\widetilde{a}q;\widetilde{a}, \widetilde{b}, \widetilde{c}; q))^2 \{ 1\!+\! A\K_{n-1}(\widetilde{c}q,\widetilde{c}q) \} }{\kappa_{n-1}(\widetilde{c}q, \widetilde{a}q)} \\
& -\frac{2 AB P_n(\widetilde{c}q;\widetilde{a}, \widetilde{b}, \widetilde{c}; q) P_n(\widetilde{a}q;\widetilde{a}, \widetilde{b}, \widetilde{c}; q)\K_{n-1}(\widetilde{c}q,\widetilde{a}q)}
{\kappa_{n-1}(\widetilde{c}q, \widetilde{a}q)}
\end{split}
\end{equation*}\end{small}
where $\widetilde{d}_n^2$ denotes the squared norm of the $n$-th big $q$-Jacobi polynomials
(see Table 1 in \cite{RC}) and
\begin{equation}\label{Kab)}\begin{split}
\kappa_{m}(s,t)& =1+A\K_{m}(s,s)+ B\K_{m}(t,t)\\
& +AB\left\{\K_{m}(s,s)\K_{m}(t,t)-(\K_{m}(s,t))^2\right\},\end{split}
\end{equation}
being $\K_{m}(s,t)$ the kernels defined by \refe{kerbigqjac}.
They can be written \cite{RR} as 
\begin{equation}
\begin{split}
P_n^{A,B}(z,\widetilde{a},\widetilde{b},\widetilde{c};q)=P_n(z,\widetilde{a},\widetilde{b},\widetilde{c};q)_q&-AP_n^{A,B}(\widetilde{c}q,\widetilde{a},\widetilde{b},\widetilde{c};q)\K_{n-1}(z,\widetilde{c}q)\\
&-BP_n^{A,B}(\widetilde{a}q,\widetilde{a},\widetilde{b},\widetilde{c};q)\K_{n-1}(z,\widetilde{a}q).
\label{repbigqjac}
\end{split}
\end{equation}
The big $q$-Jacobi-Krall-type polynomials satisfy the TTRR \cite{RC}
\begin{equation*}
z P_n^{A,B}(z,\widetilde{a},\widetilde{b},\widetilde{c};q)=\widetilde{\alpha}_n^{A,B}P_{n+1}^{A,B}(z,\widetilde{a},\widetilde{b},\widetilde{c};q)+
\widetilde{\beta}_n^{A,B}P_{n}^{A,B}(z,\widetilde{a},\widetilde{b},\widetilde{c};q)
+\widetilde{\gamma}_n^{A,B}P_{n-1}^{A,B}(z,\widetilde{a},\widetilde{b},\widetilde{c};q)
\end{equation*} 
where 
\begin{small}
\begin{equation}\label{bigqjacTTRRcoeff}
\begin{split}
\widetilde{\alpha}_n^{A,B} &= 1, \\
\widetilde{\beta}_n^{A,B} &=   \widetilde{\beta}_n
-A \left(\frac{P_n^{A,B}(\widetilde{c}q,\widetilde{a},\widetilde{b},\widetilde{c};q) P_{n-1}(\widetilde{c}q,\widetilde{a},\widetilde{b},\widetilde{c};q)}{\widetilde{d}_{n-1}^2}-
\frac{P_{n+1}^{A,B}(\widetilde{c}q,\widetilde{a},\widetilde{b},\widetilde{c};q)P_n(\widetilde{c}\widetilde{c}q,\widetilde{a},\widetilde{b},\widetilde{c};q) }{\widetilde{d}_{n}^2}\right) \\
&-B \left(\frac{P_n^{A,B}(\widetilde{a}q,\widetilde{a},\widetilde{b},\widetilde{c};q) P_{n-1}(\widetilde{a}q,\widetilde{a},\widetilde{b},\widetilde{c};q) }{\widetilde{d}_{n-1}^2}-
\frac{P_{n+1}^{A,B}(\widetilde{a}q,\widetilde{a},\widetilde{b},\widetilde{c};q) P_n(\widetilde{a}q,\widetilde{a},\widetilde{b},\widetilde{c};q) }{\widetilde{d}_{n}^2}\right), \\
\widetilde{\gamma}_n^{A,B} &= \widetilde{\gamma}_n\frac{1\!+\!\widetilde{\Delta}_n^{A,B}}{1\!+\!\widetilde{\Delta}_{n-1}^{A,B}},\,\, \widetilde{\Delta}_n^{A,B}=
\frac{ A P_n^{A,B}(\widetilde{c}q,\widetilde{a},\widetilde{b},\widetilde{c};q) P_n(\widetilde{c}q,\widetilde{a},\widetilde{b},\widetilde{c};q) }{\widetilde{d}_{n}^2}\!+\!
\frac{B P_n^{A,B}(\widetilde{a}q,\widetilde{a},\widetilde{b},\widetilde{c};q) P_n(\widetilde{a}q,\widetilde{a},\widetilde{b},\widetilde{c};q)}{\widetilde{d}_{n}^2}.
\end{split}
\end{equation}
\end{small}

The monic dual $q$-Hahn polynomials are defined by \cite{ks}
\begin{equation}\label{dualqh}
\begin{split}
R_n(s)_q:=R_{n}(x(s),\gamma,\delta,N;q) = (\gamma q,q^{-N};q)_n\,
{}_{3}\varphi_2 \left(\ba{c} q^{-n},q^{-s},
\gamma\delta q^{s+1}  \\ \gamma q,q^{-N}  \ea
\,\bigg|\, q \,,\, q \right).
\end{split}
\end{equation}
They are  orthogonal with respect to the positive weight function
\cite[(14.7.2)]{ks} supported on the points $x(s)=q^{-s}+\gamma\delta q^{s+1},$ $s=0,1,...,N$, for $0<\gamma q<1$, $0<\delta q<1$ or for $\gamma>q^{-N}$, $\delta>q^{-N}$, i.e., 
\begin{equation}\label{dualqhahnorth}
\sum_{s=0}^NR_n(s)_qR_m(s)_q\widetilde{\rho}(s)\Delta x(s-\mbox{$\frac 12$})=\widetilde{d}_n^2,
\quad \Delta x(s-\mbox{$\frac 12$})=(-\kappa_q)q^{-s}(1-\gamma\delta q^{2s+1})
\end{equation}
\begin{equation*}
\begin{split}
\widetilde{\rho}(s)&=\frac{(\gamma q)^N(\delta q;q)_N}{(\gamma\delta q^2;q)_N}
\frac{q^{Ns-(^s_2)}}{(-\kappa_q)(1-\gamma\delta q)(-\gamma)^s}
\frac{(\gamma q,\gamma\delta q,q^{-N};q)_s}{(q,\gamma\delta q^{N+2},\delta q;q)_s},\\
\widetilde{d}_n^2&=(\gamma\delta q)^n(q,q^{-N},\gamma q,\delta^{-1}q^{-N};q)_n.
\end{split}
\end{equation*}

Finally, we introduce the monic $q$-Hahn polynomials \cite{ks}
\begin{equation}\label{qhahn}
\begin{split}
h_n(s)_q:=h_{n}^{\widetilde{\alpha},\widetilde{\beta}}(x(s);N|q)=
\frac{(q^{-N},\widetilde{\alpha}q;q)_n}{(\widetilde{\alpha}\widetilde{\beta} q^{n+1};q)_n}
{}_{3}\varphi_2 \left(\ba{c} q^{-n},\widetilde{\alpha}\widetilde{\beta} q^{n+1}, x(s),
 \\ q^{-N},\widetilde{\alpha}q  \ea
\,\bigg|\, q \,,\, q \right),
\end{split}
\end{equation}
which are orthogonal with respect to a positive weight function
\cite[(14.6.2)]{ks} supported on the points $x(s)=q^{-s},$ $s=0,1,...,N$, 
for $0<\alpha q<1$, $0<\beta q<1$ or for $\alpha>q^{-N}$, $\beta>q^{-N}$, i.e.,
\begin{equation}\label{qhahnorth}
\sum_{s=0}^Nh_n(s)_qh_m(s)_q\widetilde{\rho}(s)\Delta x(s-\mbox{$\frac 12$})=\widetilde{d}_n^2,
\end{equation}
where $\Delta x(s-\mbox{$\frac 12$})=-\kappa_q q^{-s}$,
\begin{equation*}
\begin{split}
\widetilde{\rho}(s)&=\frac{(\alpha\beta)^{-s}}{(-\kappa_q)}
\frac{(\alpha q)^N(\beta q; q)_N}{(\alpha\beta q^2;q)_N}\frac{(\alpha q, q^{-N};q)_s}{(q,\beta^{-1}q^{-N};q)_s},\\
\widetilde{d}_n^2&=(-\alpha q)^nq^{(^n_2)-Nn}\frac{1-\alpha\beta q}{1-\alpha\beta q^{2n+1}}
\frac{(q,\alpha q,\beta q,q^{-N},\alpha\beta q^{N+2};q)_n}{(\alpha\beta q,\alpha\beta q^{n+1},\alpha\beta q^{n+1};q)_n}.
\end{split}
\end{equation*}


\section{Limit relation between non-standard $q$-Racah and big $q$-Jacobi polynomials}
In this section we establish a limit formula from the non-standard $q$-Racah polynomials 
\refe{q-racah} to big $q$-Jacobi polynomials \refe{bigqjacobi} that preserves the orthogonality relation 
as well as the TTRR. 
\begin{theorem}\label{thm1}
Let
\begin{equation}\label{bigqjactrans}
 \widetilde{\mu}(s)=q^{N+1}\widetilde{a}q^{a+1}\Big[\frac{\mu(s+a)-c_3}{c_1}\Big], \,
q^{\alpha}=\widetilde{a},\, q^{\beta}=\widetilde{b},\, q^{a-b}=q^{-N-1}, \,q^{a+b}=\frac{\widetilde{c}}{\widetilde{a}}.
\end{equation}
Then, the following limit formula between the non-standard $q$-Racah
and big $q$-Jacobi polynomials holds
\begin{equation}\label{limitqracah-bigqjacobi}
\lim_{N\rightarrow\infty}C_nu_n^{\alpha,\beta}( \widetilde{\mu}(s),a,b)_q= 
P_n(\widetilde x(s),\widetilde{a},\widetilde{b},\widetilde{c};q),
\end{equation}
where
\begin{equation}\label{normalizationconst}
C_n=(q^{N})^{n/2} (q\widetilde{a})^n\left(\frac{\widetilde{c}}{\widetilde{a}}\right)^{\frac n2}
\kappa_q^{2n} \, \mbox{ and } \,  \{\widetilde x(s)\}=\{q^{s+1}\widetilde{c}\}_{s=0}^{\infty}\bigcup
\{q^{s+1}\widetilde{a}\}_{s=0}^{\infty}.
\end{equation}
Moreover, the orthogonality relation of the non-standard $q$-Racah polynomials 
\refe{ort-umod} becomes into the one of big $q$-Jacobi polynomials \refe{bigqjacobiorth}.
\end{theorem}

\begin{proof}
From \refe{qracahlattice} it follows that for $s=0,1,...,N$,
\begin{equation*}
\begin{split}
q^{N+1}\widetilde{a}q^{a+1}\Big[\frac{\mu(s+a)-c_3}{c_1}\Big] & =q^{N+1}\widetilde{a}\frac{q^{a+1}}{c_1}
\Big[c_1(q^{s+a}+q^{-s-a-1})+c_3-c_3\Big]\\
&=q^{N+1}\widetilde{a}\Big[q^{-s}+q^{2a}q^{s+1}\Big]=q^{N+1}\widetilde{a}
q^{-s}+\widetilde{c}q^{s+1}=\widetilde{\mu}(s).
\end{split}
\end{equation*}

Following \cite{TK} 
we remark that for certain $M$ depending on $N$ such that $M<N$, the set of points  $\{\widetilde{\mu}(s)\}_{s=0}^N$
can be written as the union of the increasing sequence of non positive points 
$$
\Big\{q\widetilde{c}+q^{N+1}\widetilde{a}, q^2\widetilde{c}+q^{N}\widetilde{a},...,
q^M\widetilde{c}+q^{N-M+2}\widetilde{a}\Big\}
$$
and the decreasing sequence of non negative points
$$
\Big\{q\widetilde{a}+q^{N+1}\widetilde{c}, q^2\widetilde{a}+q^{N}\widetilde{c}, ...,
q^{M+1}\widetilde{a}+q^{N-M+1}\widetilde{c}\Big\},
$$
which tend to the union of the sequence of negative points $\{q^{s+1}\widetilde{c}\}_{s=0}^{\infty}$
and the sequence of positive points $ \{q^{s+1}\widetilde{a}\}_{s=0}^{\infty}$ as $N\rightarrow\infty$,
i.e., to the set  $\{\widetilde x(s)\}$. Notice that it is precisely the support of the orthogonality measure
of the big $q$-Jacobi polynomials (see \refe{bigqjacobiorth}).

Next, we rewrite \refe{q-racah} by using the identity (see e.g. \cite{RR})
\begin{equation*}
(q^{s_1-s};q)_k(q^{s_1+s+\xi};q)_k=(-1)^kq^{k(s_1+\xi+\frac{k-1}{2})}\prod_{i=0}^{k-1}
\Big[\frac{\mu(s)-c_3}{c_1}-q^{-\frac{\xi}{2}}\Big(q^{s_1+i+\frac{\xi}{2}}+
q^{-s_1-i-\frac{\xi}{2}}\Big)\Big].
\end{equation*} 
In fact, setting $s_1=a, \xi=1$ and making the transformation \refe{bigqjactrans}, 
it becomes 
\begin{equation*}
\begin{split}
(q^{-s};q)_k(q^{s+2a+1};q)_k&=(-1)^k\frac{q^{k(a+1+\frac{k-1}{2})}}{q^{(N+1)k}
\widetilde{a}^kq^{k(a+1)}}\\
&\times\prod_{i=0}^{k-1}\Big[q^{N+1}\widetilde{a}q^{a+1}\frac{\mu(s+a)-c_3}{c_1}-\widetilde{c}q^{i+1}-\widetilde{a}q^{N+1-i}\Big].
\end{split}
\end{equation*}
Then, 
\begin{equation*}
\begin{split}
u_n^{\alpha,\beta}(s+a,a,b)_q  &= \Big(q^{-N}\frac{\widetilde{c}}{\widetilde{a}}\Big)^{-n/2}\frac{(q^{-N},\widetilde{b}q,\widetilde{c}q;q)_n}
{\kappa_q^{2n}(\widetilde{a}\widetilde{b}q^{n+1};q)_n}\sum_{k=0}^{n}\frac{(q^{-n}, \widetilde{a}\widetilde{b}q^{n+1};q)_k}
{(\widetilde{b}q,\widetilde{c}q,q;q)_k}q^k\\
& \times\Bigg\{\frac{1}{(q^{-N};q)_k}\frac{(-1)^kq^{k\frac{(k-1)}{2}}}{q^{(N+1)k}\widetilde{a}^k}
\prod_{i=0}^{k-1}\Big[\widetilde{\mu}(s)
-\widetilde{c}q^{i+1}-
\widetilde{a}q^{N+1-i}\Big]\Bigg\}.
\end{split}
\end{equation*}
Next we need to take the limit $N\rightarrow\infty$. Notice that 
the set $\{\widetilde{\mu}(s)\}_{s=0}^N$ becomes into the set  
$ \{\widetilde x(s)\}:=\{q^{s+1}\widetilde{c}\}_{s=0}^{\infty}\bigcup \{q^{s+1}\widetilde{a}\}_{s=0}^{\infty}$.
Then, 
\begin{equation*}
\begin{split}
\lim_{N\rightarrow\infty}C_nu_n^{\alpha,\beta}(\widetilde{\mu}(s),a,b)_q &= (-\widetilde{a})^nq^{\frac{n(n+1)}{2}}
\frac{(\widetilde{b}q,\widetilde{c}q;q)_n}
{(\widetilde{a}\widetilde{b}q^{n+1};q)_n}\\
&\times\sum_{k=0}^{\infty}\frac{(q^{-n}, \widetilde{a}\widetilde{b}q^{n+1};q)_k}
{(\widetilde{b}q,\widetilde{c}q,q;q)_k}(\widetilde{a}^{-1} \widetilde{x}(s))^{k}\prod_{i=0}^{k-1}
\left[1-\frac{\widetilde{c}q^{i+1}}{\widetilde{x}(s)}\right]
\end{split}
\end{equation*}
where $\prod_{i=0}^{k-1}\Big[1-\widetilde{c}q^{i+1}/\widetilde{x}(s)\Big]=(q\widetilde{c}/\widetilde{x}(s);q)_k,$
from where \refe{limitqracah-bigqjacobi} follows.  A similar analysis was done 
in \cite{TK} but starting from the standard $q$-Racah polynomials. 

To show that the orthogonality relation of the 
non-standard $q$-Racah polynomials becomes into the one of 
the big $q$-Jacobi polynomials we rewrite \refe{ort-umod} 
 using the transformation \refe{bigqjactrans}. This yields 
\bq\label{qracahorth}
\begin{split}
\sum_{s=0}^Nf(q^s)
\Delta  \widetilde{\mu}(s-\mbox{$\frac 12$})&
=\sum_{s=0}^{M-1}f(q^s)\Delta[q^{N+1}\widetilde{a}q^{s+\frac 12}+\widetilde{c}q^{s+\frac 12}]\\
&-\sum_{s=0}^{N-M}f(q^{N-s})\Delta[\widetilde{a}q^{s+\frac 12}+q^{N+1}\widetilde{c}q^{-s+\frac 12}]=d_n^2\delta_{mn},
\end{split}
\eq
where $M$, as before, depends on $N$, $M<N$ and
$$
f(q^s)=\frac{1}{C_nC_m}[C_n{u}_n^{\alpha, \beta}(s+a)_q] [C_m{u}_m^{\alpha, \beta}(s+a)_q] 
\frac{c_1(-\kappa_q)}{q^{N+1}\widetilde{a}q^{a+1}}\rho(s+a).
$$ 
If we take the limit $N\rightarrow\infty$ and use that 
\begin{equation}\label{Nrelation}
\begin{split}
& \lim_{N\to\infty}C_n^2d_n^2=\widetilde{d}_n^2,\quad \lim_{N\rightarrow\infty}\frac{c_1(-\kappa_q)}{q^{N+1}\widetilde{a}q^{a+1}}\rho(s+a)=(1-q)\widetilde{\rho}(\widetilde{c}q^{s+1}),\\&
\lim_{N\rightarrow\infty}\frac{c_1(-\kappa_q)}{q^{N+1}\widetilde{a}q^{a+1}}\rho(N-s+a)=(1-q)\widetilde{\rho}(\widetilde{a}q^{s+1}),
\end{split}
\end{equation}
where $\widetilde{\rho}$ and $\widetilde{d}_n$ are the weight function and 
the norm of the big $q$-Jacobi polynomials, respectively,  
then \refe{qracahorth} becomes into
\begin{equation*}
\begin{split}
&(1-q)(-\widetilde{c})\sum_{s=0}^{\infty}P_n(\widetilde{c}q^{s+1})P_m(\widetilde{c}q^{s+1})\widetilde{\rho}
(\widetilde{c}q^{s+1})q^{s+1}\\
&+(1-q)\widetilde{a}\sum_{s=0}^{\infty}P_n(\widetilde{a}q^{s+1})P_m(\widetilde{a}q^{s+1})\widetilde{\rho}(\widetilde{a}q^{s+1})q^{s+1}=\widetilde{d}_n^2\delta_{mn},
\end{split}
\end{equation*}
which is the orthogonality relation of the big $q$-Jacobi polynomials \refe{bigqjacobiorth}.
\end{proof}

To conclude this part let us show that the limit procedure stated in Theorem \ref{thm1} transforms also 
the TTRR of the non-standard $q$-Racah polynomials \refe{q-racahTTRR} into the
TTRR of the monic big $q$-Jacobi polynomials. 
Using the transformation \refe{bigqjactrans} in the TTRR \refe{q-racahTTRR}, we get
\begin{equation*}
\begin{split}
 \widetilde{\mu}(s) u_n^{\alpha,\beta}(s+a)_q=& q^{N+1}\widetilde{a}\frac{q^{a+1}}{c_1}u_{n+1}^{\alpha,\beta}(s+a)_q+q^{N+1}\widetilde{a}\frac{q^{a+1}}{c_1}[\beta_n-c_3] u_{n}^{\alpha,\beta}(s+a)_q
\\
& + q^{N+1}\widetilde{a}\frac{q^{a+1}}{c_1}\gamma_nu_{n-1}^{\alpha,\beta}(s+a)_q.
\end{split}
\end{equation*} 
Multiplying the above equality by the normalization constant $C_n$
\refe{normalizationconst}, taking the limit $N\rightarrow\infty$  
and using the relation \refe{limitqracah-bigqjacobi}, we get
\begin{equation*}
\begin{split}
&\widetilde x(s)P_n(\widetilde x(s),\widetilde{a},\widetilde{b},\widetilde{c};q)=
P_{n+1}(\widetilde x(s),\widetilde{a},\widetilde{b},\widetilde{c};q)
+\widetilde{\beta}_nP_{n}(\widetilde x(s),\widetilde{a},\widetilde{b},\widetilde{c};q)  
+\widetilde{\gamma}_nP_{n-1}(\widetilde x(s),\widetilde{a},\widetilde{b},\widetilde{c};q),
\end{split}
\end{equation*} 
where
\begin{equation}\label{betagammatilde}
\begin{split}
\widetilde{\beta}_n=\lim_{N\rightarrow\infty}q^{N+1}\widetilde{a}\frac{q^{a+1}}{c_1}[\beta_n-c_3],\quad 
\widetilde{\gamma}_n=\lim_{N\rightarrow\infty}q^{N+1}\widetilde{a}\frac{q^{a+1}}{c_1}\frac{C_n}{C_{n-1}}\gamma_n.
\end{split}
\end{equation}
Notice that
\begin{equation*}
\begin{split}
&q^{N+1}\widetilde{a}\frac{q^{a+1}}{c_1}[\beta_n-c_3]=q^{N+1}\widetilde{a}\frac{q^{a+1}}{c_1}\Bigg[
c_1(q^a+q^{-a-1})-q^{-1/2}(1+q)\kappa_q^{-2}\\
&-\frac{q^{-a-\frac 12}
(1-q^{\alpha+\beta+n+1})(1-q^{a-b+n+1})(1-q^{\beta+n+1})(1-q^{a+b+\alpha+n+1})}
{\kappa_q^2(1-q^{\alpha+\beta+2n+1})(1-q^{\alpha+\beta+2n+2})}\\
&-\frac{q^{a+\frac 12}
(1-q^{\alpha+n})(1-q^{b-a+\alpha+\beta+n})(1-q^{-a-b+\beta+n})(1-q^{n})}
{q^{-\frac 12[\alpha+\beta+2n+\alpha+\beta+2n+1]}\kappa_q^2(1-q^{\alpha+\beta+2n})(1-q^{\alpha+\beta+2n+1})}
+q^{-1/2}(1+q)\kappa_q^{-2}\Bigg],
\end{split}
\end{equation*} 
which becomes, by using the transformation \refe{bigqjactrans}, into
\begin{equation*}
\begin{split}
q^{N+1}\widetilde{a}\frac{q^{a+1}}{c_1}[\beta_n-c_3]&=q\widetilde{c}+q^{N+1}\widetilde{a}
-\frac{(1-\widetilde{a}\widetilde{b}q^{n+1})(1-q^{n-N})(1-\widetilde{b}q^{n+1})(1-\widetilde{c}q^{n+1})}{q^{-N-1}\widetilde{a}^{-1}(1-\widetilde{a}\widetilde{b}q^{2n+1})(1-\widetilde{a}\widetilde{b}q^{2n+2})}\\
&-\frac{q\widetilde{c}(1-\widetilde{a}q^{n})(1-\widetilde{a}\widetilde{b}q^{n+N+1})(1-\widetilde{a}\widetilde{b}\widetilde{c}^{-1}q^{n})(1-q^n)}{(1-\widetilde{a}\widetilde{b}q^{2n})(1-\widetilde{a}\widetilde{b}q^{2n+1})}.
\end{split}
\end{equation*} 
Finally, by taking the limit as $N\rightarrow\infty$, we obtain
\begin{equation*}
\begin{split}
\widetilde{\beta}_n=\lim_{N\rightarrow\infty}q^{N+1}\widetilde{a}\frac{q^{a+1}}{c_1}[\beta_n-c_3]&=q\widetilde{c}+\widetilde{a}q^{n+1}
\frac{(1-\widetilde{a}\widetilde{b}q^{n+1})(1-\widetilde{b}q^{n+1})(1-\widetilde{c}q^{n+1})}{(1-\widetilde{a}\widetilde{b}q^{2n+1})(1-\widetilde{a}\widetilde{b}q^{2n+2})}\\
&-q\widetilde{c}\frac{(1-\widetilde{a}q^{n})(1-\widetilde{a}\widetilde{b}\widetilde{c}^{-1}q^{n})(1-q^n)}{(1-\widetilde{a}\widetilde{b}q^{2n})(1-\widetilde{a}\widetilde{b}q^{2n+1})},
\end{split}
\end{equation*} 
which is equivalent to $\widetilde{\beta}_n$ in \refe{tildebetan}.
In a complete analogous way, one can obtain 
\begin{equation*}
\begin{split}
\widetilde{\gamma}_n= \lim_{N\rightarrow\infty}  q^{N+1}\widetilde{a}\frac{q^{a+1}}{c_1}\frac{C_n}{C_{n-1}}\gamma_n&=-\widetilde{a}\widetilde{c}q^{n+1}(1\!-\!\widetilde{a}q^{n})(1\!-\!\widetilde{b}q^{n})(1\!-\!\widetilde{c}q^{n})\\
&\times\frac{(1-q^n)(1-\widetilde{a}\widetilde{b}q^{n})(1-\widetilde{a}\widetilde{b}\widetilde{c}^{-1}q^{n})}{(1-\widetilde{a}\widetilde{b}q^{2n-1})(1-\widetilde{a}\widetilde{b}q^{2n})^2(1-\widetilde{a}\widetilde{b}q^{2n+1})},
\end{split}
\end{equation*} 
which is the coefficient $\widetilde{\gamma}_n$  of the TTRR of 
the big $q$-Jacobi polynomials \refe{tildebetan}.

A similar analysis can be done but starting with the standard $q$-Racah polynomials \cite[page 422]{ks}.


\section{Limit relation between non-standard $q$-Racah and $q$-Hahn polynomials}

In this section we consider two more limit cases: namely, the limits to 
$q$-Hahn \refe{dualqh} and dual $q$-Hahn \refe{qhahn} polynomials. 
In these two cases the situation is more simple
since these two families are finite families (the orthogonality measure is supported 
on a finite set) as the non-standard $q$-Racah polynomials. 

The limit formula between non-standard $q$-Racah and dual $q$-Hahn polynomials
is stated in the following theorem.
\begin{theorem}\label{thm4}
Let 
\bq\label{dualqhahntrans}
\widetilde\mu(s)=q^{a+1}\frac{\mu(s+a)-c_3}{c_1},\quad q^{a-b}=q^{-N-1},\quad
q^{\beta}=\gamma, \quad q^{2a}=\gamma\delta. 
\eq
Then
\[
\lim_{q^{\alpha}\rightarrow 0}
C_n u_n^{\alpha,\beta}(\widetilde\mu(s) ,a,b)= R_{n}(x(s),\gamma,\delta,N;q),\quad C_n=(\gamma\delta q)^{n/2}\kappa_q^{2n},
\quad x(s)=q^{-s}+\gamma\delta q^{s+1}.
\]
Moreover, the orthogonality relation of the non-standard $q$-Racah polynomials 
\refe{ort-umod}
becomes into the one of dual $q$-Hahn polynomials \refe{dualqhahnorth}.
\end{theorem}

\begin{proof} We present here only the sketch of the proof. 
First of all, notice that by using the transformation \refe{dualqhahntrans} 
it follows from \refe{qracahlattice} that
\[
\begin{split}
\widetilde\mu(s)=
q^{a+1}\frac{\mu(s+a)-c_3}{c_1}&= q^{-s}+\gamma\delta q^{s+1}=x(s).
\end{split}
\] 
Taking into account that the weight function of the dual 
$q$-Hahn polynomials is also supported on a finite set of points, 
then it is straightforward to see that the orthogonality relation 
\refe{dualqhahnorth} can be derived from 
the one of the non-standard $q$-Racah polynomials \refe{ort-umod}
by taking the limit procedure defined in Theorem \ref{thm4}.
\end{proof}

Moreover, applying the same transformation to the TTRR for non-standard $q$-Racah polynomials 
\refe{q-racahTTRR} and taking the limit $q^ \alpha\to0$, we get   
$$
x(s)R_{n}(s)_q=R_{n+1}(s)_q+\widetilde{\beta}_nR_{n}(s)_q
+\widetilde{\gamma}_nR_{n-1}(s)_q,
$$
where
\[
\begin{split}
\widetilde{\beta}_n&=\lim_{q^{\alpha}\rightarrow0}\frac{q^{a+1}}{c_1}[\beta_n-c_3]= 1+\gamma\delta q-(1-\gamma q^{n+1})(1-q^{n-N})
-\gamma q(\delta-q^{n-N-1})(1-q^{n}),\\
\widetilde{\gamma}_n&=\lim_{q^{\alpha}\rightarrow0}\frac{q^{a+1}}{c_1}\frac{C_n}{C_{n-1}}\gamma_n=
\gamma q(1-q^{n})(\delta-q^{n-N-1})(1-q^{n-N-1})(1-\gamma q^{n}),
\end{split}
\]
which is the TTRR for the dual $q$-Hahn polynomials.

 
\begin{theorem}\label{thmqhahn}
Let
\bq\label{qhahntrans}
\widetilde\mu(s)= q^{a+1}\frac{\mu(s+a)}{c_1}, \, q^{\alpha}=\widetilde{\beta},\,
 q^{\beta}=\widetilde{\alpha},\, q^{-b}=q^{-N-1-a}. 
\eq
Then,
\[
\lim_{q^a\rightarrow 0}C_nu_n^{\alpha,\beta}(\widetilde\mu(s),a,b)
=h_{n}^{\widetilde{\alpha},\widetilde{\beta}}(x(s);N|q),\quad 
C_n=q^{\frac{n}{2}(2a+1)}\kappa_q^{2n}, \quad x(s)=q^{-s}.
\]
\end{theorem}
\begin{proof} We here only sketch the proof. 
Notice that the limit procedure stated in Theorem \ref{thmqhahn} yields
$\lim_{q^a\rightarrow0}\widetilde\mu(s)=q^{-s}=x(s)
$. 
Moreover, the orthogonality relation of the $q$-Hahn
polynomials \refe{qhahnorth} easily follows from the orthogonality relation of
the non-standard $q$-Racah polynomials \refe{ort-umod} by taking the limit defined 
in the Theorem \ref{thmqhahn}.
\end{proof}

Finally, let us mention that by using the transformation stated in Theorem \ref{thmqhahn},
the recurrence relation for the $q$-Racah polynomials 
\refe{q-racahTTRR} transforms into the TTRR for $q$-Hahn polynomials \cite{ks}:
\[
x(s)h_{n}(s)_q=h_{n+1}(s)_q+\widetilde{\beta}_nh_{n}(s)_q
+\widetilde{\gamma}_nh_{n-1}(s)_q,
\]
\[
\begin{split}
\widetilde{\beta}_n =\lim_{q^a\rightarrow0} \frac{q^{a+1}}{c_1}\beta_n&=  1-\frac{(1-\widetilde{\alpha}\widetilde{\beta}q^{n+1})
(1-\widetilde{\alpha}q^{n+1})(1-q^{n-N})}{(1-\widetilde{\alpha}\widetilde{\beta}q^{2n+1})(1-\widetilde{\alpha}\widetilde{\beta}q^{2n+2})}\\
&+\frac{\widetilde{\alpha}q^{n-N}(1-q^n)(1-\widetilde{\beta}q^n)(1-\widetilde{\alpha}\widetilde{\beta}q^{n+N+1})}{(1-\widetilde{\alpha}\widetilde{\beta}q^{2n})(1-\widetilde{\alpha}\widetilde{\beta}q^{2n+1})},\\
\widetilde{\gamma}_n=\lim_{q^a\rightarrow0}\frac{q^{a+1}}{c_1}\frac{C_n}{C_{n-1}}\gamma_n&= 
-\widetilde{\alpha}q^{n-N}\frac{(1-\widetilde{\alpha}\widetilde{\beta}q^{n})(1-\widetilde{\alpha}q^{n})(1-\widetilde{\beta}q^{n})(1-\widetilde{\alpha}\widetilde{\beta}q^{n+N+1})}{(1-\widetilde{\alpha}\widetilde{\beta}q^{2n-1})(1-\widetilde{\alpha}\widetilde{\beta}q^{2n})^2
(1-\widetilde{\alpha}\widetilde{\beta}q^{2n+1})}\\ 
&\times (1-q^n)(1-q^{n-N}).
\end{split}
\]

\section{The limit relation for the Krall-type polynomials}

As a consequence of the limit formula between
non-standard $q$-Racah and big $q$-Jacobi polynomials (see Theorem \ref{thm1})
one can obtain the big $q$-Jacobi-Krall-type polynomials from
the $q$-Racah-Krall-type polynomials. 

In order to get this kind of limit formula for the Krall-type polynomials, 
we apply the transformation \refe{bigqjactrans} to the formula \refe{repforqracah}.  So,  
multiplying \refe{repforqracah} by $C_n^2$ in \refe{bigqjactrans} we have the expression 
 
\begin{equation*}
C_nu_n^{\alpha,\beta,A,B}(s)_q=C_nu_n^{\alpha,\beta}(s)_q-AC_nu_n^{\alpha,\beta,A,B}(a)_q\K_{n-1}^{\alpha,\beta}(s,a)-BC_nu_n^{\alpha,\beta,A,B}(b-1)_q\K_{n-1}^{\alpha,\beta}(s,b-1).
\end{equation*}
But $\lim_{N\to\infty}C_nu_n^{\alpha,\beta}(s)_q= P_n(z,\widetilde{a},\widetilde{b},\widetilde{c};q)$
\refe{limitqracah-bigqjacobi} where $C_n$ is defined in \refe{normalizationconst}. The kernel \refe{rez13} becomes
\[
\begin{split}
\K_n^{\alpha,\beta}(s_1, s_2):&=\lim_{N\to\infty}\dst\sum_{k=0}^{n}
{\dst\frac{\Big[C_ku_k^{\alpha,\beta}(s_1)_q\Big]\Big[C_ku_k^{\alpha,\beta}(s_2)_q\Big]}{C_k^2d_k^2}}\\
&=\dst\sum_{k=0}^{n}{\dst\frac{P_n(z_1,\widetilde{a},\widetilde{b},\widetilde{c};q)P_n(z_2,\widetilde{a},\widetilde{b},\widetilde{c};q)}{\widetilde{d}_k^2}}=\K_n(z_1,z_2)
\end{split}
\]
whereas
\[ \lim_{N\to\infty}C_nu_n^{\alpha, \beta}(a)_q= P_n(\widetilde{c}q,\widetilde{a},\widetilde{b},\widetilde{c};q),
\quad \lim_{N\to\infty}C_nu_n^{\alpha, \beta}(b-1)_q= P_n(\widetilde{a}q,\widetilde{a},\widetilde{b},\widetilde{c};q).
\]
Then, straightforward calculations lead to the following expressions for the values of 
the $q$-Racah-Krall-type polynomials at the points $a$ and $b-1$ 
\begin{eqnarray*}
\lim_{N\to\infty}C_nu_n^{\alpha, \beta,A,B}(a)_q= P_n^{A,B}(\widetilde{c}q,\widetilde{a},\widetilde{b},\widetilde{c};q),
\quad \lim_{N\to\infty}C_nu_n^{\alpha, \beta}(b-1)_q= P_n^{A,B}(\widetilde{a}q,\widetilde{a},\widetilde{b},\widetilde{c};q),
\end{eqnarray*}
and the representation formula for the $q$-Racah-Krall-type polynomials \refe{repforqracah}
transforms into the representation formula for the big $q$-Jacobi polynomials
defined in \refe{repbigqjac},
\begin{equation*}
P_n^{A,B}(z,\widetilde{a},\widetilde{b},\widetilde{c};q)\!\!=\!\!P_n(z,\widetilde{a},\widetilde{b},\widetilde{c};q)_q-AP_n^{A,B}(\widetilde{c}q,\widetilde{a},\widetilde{b},\widetilde{c};q)\K_{n-1}(z,\widetilde{c}q)-BP_n^{A,B}(\widetilde{a}q,\widetilde{a},\widetilde{b},\widetilde{c};q)\K_{n-1}(z,\widetilde{a}q).
\end{equation*}

Notice also that using the transformation \refe{bigqjactrans}
into the orthogonality relation \refe{ort-umodi} yields to
\bq\label{qracahorthmodi}
\begin{split}
&\sum_{s = 0 }^{N} \Big[C_n{u}_n^{\alpha, \beta, A,B}(s)_q\Big] \Big[C_m{u}_m^{\alpha, \beta, A,B}(s)_q \Big]
\rho(s)\Delta  \widetilde{\mu}(s-\mbox{$\frac 12$}) + A C_n{u}_n^{\alpha, \beta, A,B}(a)_q
C_m{u}_m^{\alpha, \beta, A,B}(a)_q \\ 
& +BC_n{u}_n^{\alpha, \beta, A,B}(b\!-\!1)_q
C_m{u}_m^{\alpha, \beta, A,B}(b\!-\!1)_q 
=\sum_{s=0}^{M-1}f(q^s)\Delta[q^{N+1}\widetilde{a}q^{s+\frac 12}+\widetilde{c}q^{s+\frac 12}]\\
&-\sum_{s=0}^{N-M}f(q^{N-s})\Delta[\widetilde{a}q^{s+\frac 12}+q^{N+1}\widetilde{c}q^{-s+\frac 12}]+ A C_n{u}_n^{\alpha, \beta, A,B}(a)_q
C_m{u}_m^{\alpha, \beta, A,B}(a)_q \\ & +BC_n{u}_n^{\alpha, \beta, A,B}(b\!-\!1)_q
C_m{u}_m^{\alpha, \beta, A,B}(b\!-\!1)_q = C_nC_m \delta_{n,m} (d_n^{A,B})^2,
\end{split}
\eq
where $M$ depends on $N$ such that $M<N$ and
$$
f(q^s)=[C_n{u}_n^{\alpha, \beta,A,B}(s+a)_q] [C_m{u}_m^{\alpha, \beta,A,B}(s+a)_q] 
\frac{c_1(-\kappa_q)}{q^{N+1}\widetilde{a}q^{a+1}}\rho(s+a).
$$ 
If we now take the limit $N\rightarrow\infty$ and use \refe{normalizationconst} and \refe{Nrelation} we get
\begin{equation*}
\begin{split}
\lim_{N\to\infty}\left(C_nd_n^{A,B}\right)^2 = (\widetilde{d}_n^{A,B})^2.
\end{split}
\end{equation*}
Using now \refe{dn-u2m}, expression \refe{qracahorthmodi} becomes into
\begin{equation*}
\begin{split}
&(1-q)(-\widetilde{c})\sum_{s=0}^{\infty}P_n(\widetilde{c}q^{s+1})P_m(\widetilde{c}q^{s+1})\widetilde{\rho}
(\widetilde{c}q^{s+1})q^{s+1}\\
&+(1-q)\widetilde{a}\sum_{s=0}^{\infty}P_n(\widetilde{a}q^{s+1})P_m(\widetilde{a}q^{s+1})\widetilde{\rho}(\widetilde{a}q^{s+1})q^{s+1}+ A P_n^{A,B}(\widetilde{c}q)_q
P_m^{A,B}(\widetilde{c}q)_q \\ & +BP_n^{A,B}(\widetilde{a}q)_q
P_m^{A,B}(\widetilde{a}q)_q =(\widetilde{d}_n^{A,B})^2\delta_{mn},
\end{split}
\end{equation*}
which is the orthogonality relation of the big $q$-Jacobi-Krall-type polynomials \refe{bigqjacobimodiorth}.

To conclude this section we last deal with the TTRR
of the $q$-Racah-Krall-type polynomials \refe{q-racahmodTTRR}
considering the transformation defined in \refe{bigqjactrans}
which leads to
\begin{equation*}
\begin{split}
 \widetilde{\mu}(s)C_n u_n^{\alpha,\beta,A,B}(s+a)_q&= q^{N+1}\widetilde{a}\frac{q^{a+1}}{c_1}\alpha_n^{A,B}C_nu_{n+1}^{\alpha,\beta,A,B}(s+a)_q\\
 &+q^{N+1}\widetilde{a}\frac{q^{a+1}}{c_1}[\beta_n^{A,B}-c_3] C_nu_{n}^{\alpha,\beta,A,B}(s+a)_q
\\
& + q^{N+1}\widetilde{a}\frac{q^{a+1}}{c_1}\gamma_n^{A,B}C_{n}u_{n-1}^{\alpha,\beta,A,B}(s+a)_q.
\end{split}
\end{equation*} 
taking the limit $N\rightarrow\infty$  
and using the relation \refe{limitqracah-bigqjacobi}, we get
\begin{equation*}
\begin{split}
&\widetilde x(s)P_n^{A,B}(\widetilde x(s),\widetilde{a},\widetilde{b},\widetilde{c};q)=
P_{n+1}^{A,B}(\widetilde x(s),\widetilde{a},\widetilde{b},\widetilde{c};q)
+\widetilde{\beta}_n^{A,B}P_{n}^{A,B}(\widetilde x(s),\widetilde{a},\widetilde{b},\widetilde{c};q)  
+\widetilde{\gamma}_n^{A,B}P_{n-1}^{A,B}(\widetilde x(s),\widetilde{a},\widetilde{b},\widetilde{c};q),
\end{split}
\end{equation*} 
where
\begin{equation*}
\begin{split}
\widetilde{\beta}_n^{A,B}&=\lim_{N\rightarrow\infty}q^{N+1}\widetilde{a}\frac{q^{a+1}}{c_1}[\beta_n^{A,B}-c_3]=\lim_{N\rightarrow\infty}\Big\{q^{N+1}\widetilde{a}\frac{q^{a+1}}{c_1}[\beta_n-c_3]\\
&-A \Big(q^{N+1}\widetilde{a}\frac{q^{a+1}}{c_1}\frac{C_{n-1}}{C_n}\frac{C_n{u}_n^{\alpha, \beta, A,B}(a)_q C_{n-1}u_{n-1}^{\alpha, \beta }(a)_q}{C_{n-1}^2{d}_{n-1}^2}\\
&-q^{N+1}\widetilde{a}\frac{q^{a+1}}{c_1}\frac{C_{n}}{C_{n+1}}\frac{C_{n+1}{u}_{n+1}^{\alpha, \beta, A,B}(a)_qC_nu_{n}^{\alpha, \beta }(a)_q }{C_n^2{d}_{n}^2}\Big) \\
&-B\Big(q^{N+1}\widetilde{a}\frac{q^{a+1}}{c_1}\frac{C_{n-1}}{C_n}\frac{C_n{u}_n^{\alpha, \beta, A,B}(b-1)_q C_{n-1}u_{n-1}^{\alpha, \beta }(b-1)_q }{C_{n-1}^2{d}_{n-1}^2}\\
&-q^{N+1}\widetilde{a}\frac{q^{a+1}}{c_1}\frac{C_{n}}{C_{n+1}}\frac{C_{n+1}{u}_{n+1}^{\alpha, \beta, A,B}(b-1)_q C_nu_{n}^{\alpha, \beta }(b-1)_q }{C_n^2{d}_{n}^2}\Big)\Big\}, \\
\widetilde{\gamma}_n^{A,B} &=\lim_{N\rightarrow\infty}q^{N+1}\widetilde{a}\frac{q^{a+1}}{c_1}\frac{C_n}{C_{n-1}}\gamma_n^{A,B}= \lim_{N\rightarrow\infty}q^{N+1}\widetilde{a}\frac{q^{a+1}}{c_1}\frac{C_n}{C_{n-1}}\gamma_n\frac{1\!+\!\Delta_n^{A,B}}{1\!+\!\Delta_{n-1}^{A,B}},\\
& \Delta_n^{A,B}=
\frac{ A C_n{u}_n^{\alpha, \beta, A,B}(a)_q C_nu_{n}^{\alpha, \beta }(a)_q }{C_n^2{d}_{n}^2}\!+\!
\frac{B C_n{u}_n^{\alpha, \beta, A,B}(b\!-\!1)_q C_nu_{n}^{\alpha, \beta}(b\!-\!1)_q}{C_n^2{d}_{n}^2}.
\end{split}
\end{equation*}
Notice from \refe{betagammatilde} and $q^{N+1}\widetilde{a}\frac{q^{a+1}}{c_1}\frac{C_{n-1}}{C_n}=1$ that
\begin{equation*}
\begin{split}
\widetilde{\beta}_n^{A,B}&=\widetilde{\beta}_n-A \Big(\frac{P_n^{A,B}(\widetilde{c}q)_q P_{n-1}(\widetilde{c}q)_q}{\widetilde{d}_{n-1}^2}
-\frac{P_{n+1}^{A,B}(\widetilde{c}q)_qP_{n}(\widetilde{c}q)_q }{\widetilde{d}_{n}^2}\Big) \\
&-B \left(\frac{P_n^{A,B}(\widetilde{a}q)_q P_{n-1}(\widetilde{a}q)_q }{\widetilde{d}_{n-1}^2}
-\frac{P_{n+1}^{A,B}(\widetilde{a}q)_q P_{n}(\widetilde{a}q)_q }{\widetilde{d}_{n}^2}\right), \\
\widetilde{\gamma}_n^{A,B} &= \widetilde{\gamma}_n\frac{1\!+\!\widetilde{\Delta}_n^{A,B}}{1\!+\!\widetilde{\Delta}_{n-1}^{A,B}},\,\, \widetilde{\Delta}_n^{A,B}=
\frac{ A P_n^{A,B}(\widetilde{c}q)_q P_{n}(\widetilde{c}q)_q }{\widetilde{d}_{n}^2}\!+\!
\frac{B P_n^{A,B}(\widetilde{a}q)_q P_{n}(\widetilde{a}q)_q}{\widetilde{d}_{n}^2}
\end{split}
\end{equation*}
which are the the coefficients of the TTRR 
for the big $q$-Jacobi-Krall-type polynomials (see \refe{bigqjacTTRRcoeff}).

Since the limit relation from the non-standard $q$-Racah-Krall polynomials to $q$-Hahn polynomials
is quite similar to this one, we will omit it here.

\section*{Concluding remarks}
In the present work we have presented some limit formulas
from the non-standard $q$-Racah polynomials to other families 
of $q$-polynomials of the $q$-Askey scheme \cite{ks} 
and from the $q$-Racah Krall-type polynomials to big $q$-Jacobi-Krall-type polynomials
such that the orthogonality property remains present while the limit is approached.
Also we show that under these limits the TTRR of the non-standard $q$-Racah 
polynomials becomes into the TTRR of the corresponding families of
$q$-Hahn, dual $q$-Hahn and big $q$-Jacobi polynomials. 
Since the non-standard $q$-Racah polynomials $u_n^{\alpha,\beta}(s)_q$ are 
multiples of the standard $q$-Racah polynomials $R_n(x({s-a});q^\beta,q^\alpha,q^{a-b},q^{a+b}|q)$,
$b-a=N$ \cite[page 422]{ks}  (see \cite{RYR} for more details), then
the orthogonality property of  the standard $q$-Racah polynomials \cite[Eq. (14.2.2) page 422]{ks}
becomes into the one of the big $q$-Jacobi polynomials, as well as, the
TTRR of  the standard $q$-Racah polynomials \cite[Eq. (14.2.3) page 423]{ks}
becomes into the TTRR of the big $q$-Jacobi polynomials.  In such a way we have completed 
the work by Koornwinder \cite{TK} and extend it to the non-standard $q$-Racah 
polynomials introduced in \cite{NSU}  as well as to the corresponding Krall-type polynomials obtained 
via the addition of two mass points to the weight function of 
this polynomial \cite{RC,RR}.

\section*{Acknowledgements:} This work was partially supported by MTM2009-12740-C03-02
(Ministerio de Econom\'\i a y Competitividad), FQM-262, FQM-4643, FQM-7276 (Junta de Andaluc\'\i a),
Feder Funds (European Union).  
The second author is supported by a grant from T\"{U}B\.{I}TAK, the Scientific and
Technological Research Council of Turkey. This research has been done during the stay
of the second author at the Universidad de Sevilla. She also
thanks to the Departamento de An\'alisis Matem\'atico of the Universidad de Sevilla
and IMUS for their kind hospitality. 


\end{document}